\documentclass{siamltex}
\usepackage{amsmath,amssymb,color,bbm}
\usepackage{graphicx}
\usepackage{cite}
\usepackage{fullpage}

\numberwithin{equation}{section} 
\usepackage{tikz}\usetikzlibrary{arrows}

\newcommand{\R}{\mathbb{R}} 

\def\Span{\mathop{\rm span}\nolimits}
\def\tree{\mathop{\rm Tree}\nolimits}
\def\cycle{\mathop{\rm Cycle}\nolimits}
\def\prim{\mathop{\rm prim}\nolimits}

\newtheorem{thm}{Theorem}[section]
\newtheorem{lem}[thm]{Lemma}
\newtheorem{prop}[thm]{Proposition}
\newtheorem{define}[thm]{Definition}
\newtheorem{example}[thm]{Example}
\newtheorem{remark}[thm]{Remark}

\newcommand{\g}{\gamma}

\newcommand{\1}{\mathbf{1}}
\newcommand{\0}{\mathbf{0}}
\newcommand{\av}[1]{\left|{#1}\right|}
\newcommand{\ip}[2]{\left\langle{{#1}},{{#2}}\right\rangle}

\renewcommand{\l}{\left}
\newcommand{\q}{\quad}
\renewcommand{\r}{\right}

\newcommand{\G}{{\Gamma}}
\renewcommand{\phi}{\varphi}
\renewcommand{\L}{\mathcal{L}}

\newcommand{\detred}{\det_{\mathsf{red}}}
\newcommand{\mcZ}{{\mathcal{Z}}}

\newcommand{\qed}{\hfill{\ensuremath{\blacksquare}}}

\begin{document}
\title{Graph Homology and Stability of Coupled Oscillator Networks}
\date{\today}

\author{Jared C. Bronski\and Lee DeVille\and Timothy
  Ferguson\footnote{Department of Mathematics, University of Illinois,
    1409 W. Green St., Urbana, IL 61801 USA}}

\maketitle

\begin{abstract}
  There are a number of models of coupled oscillator networks where
  the question of the stability of fixed points reduces to calculating
  the index of a graph Laplacian. Some examples of such models include
  the Kuramoto and Kuramoto--Sakaguchi equations as well as the swing
  equations, which govern the behavior of generators coupled in an
  electrical network. We show that the index calculation can be
  related to a dual calculation which is done on the first homology
  group of the graph, rather than the vertex space. We also show that
  this representation is computationally attractive for relatively
  sparse graphs, where the dimension of the first homology group is
  low, as is true in many applications.  We also give explicit
  formulae for the dimension of the unstable manifold to a
  phase-locked solution for graphs containing one or two loops.  As an
  application, we present some novel results for the Kuramoto model
  defined on a ring and compute the longest possible edge length for a
  stable solution.
\end{abstract}

\section{Introduction}

Let $G = (V,E,\Gamma = \{\gamma_{v,w}\}_{v,w\in V})$ be a weighted
graph.  The Laplacian matrix\footnote{Our Laplacian is the negative of
  the standard definition, but the reasons for this will be clear
  below} of $G$ is the $|V|\times|V|$ matrix whose components are
\begin{equation*}
  (\L_G)_{vw} = \begin{cases} \gamma_{vw},&v\neq w,\\ -\sum_u \gamma_{vu},& v=w.\end{cases}
\end{equation*}

A classical result is that if $\gamma_{vw}\ge 0$, then $\L_G$ is
negative semidefinite, with the number of zero eigenvalues equal to
the number of connected components of the graph $G$.  If some of the
$\gamma_{vw}$ are negative, then $\L_G$ can have positive eigenvalues
as well.  In~\cite{Bronski.DeVille.14, BDK-SIMAX.15}, the authors
consider the problem where some subset of the $\gamma_{vw}$ are
negative, and show upper and lower bounds on the number of positive
eigenvalues of $\L_G$ that depend only on the ``sign topology'' of
the graph, i.e. the pattern of which edges are positive and which are
negative, regardless of their weights.  

More specifically, the main result of~\cite{Bronski.DeVille.14} is:
assume that the graph is symmetric ($\gamma_{(v,w)} =
\gamma_{(w,v)})$, and define $G_+$ (resp.~$G_-$) to be the subgraphs
of $G$ containing only edges with positive (resp.~negative)
weights. Let $|G|$ be the number of components of a graph $G$.  If
$n_+(G)$ is the number of positive eigenvalues of $\L_G$, then
\begin{equation}\label{eq:bound}
  |G_+| -1 \le n_+(G) \le |V|-|G_-|.
\end{equation}
Moreover, one can show that these bounds are saturated, i.e. for any
integer in that range, there is some choice of weights that give
exactly that many positive eigenvalues.  (In fact, one can say much
more about the genericity of the sets of weights that give these
values, see~\cite{BDK-SIMAX.15} for more detail.)

It follows from the formula above that if $G$ is a tree, then the
number of positive eigenvalues is exactly the number of negatively
weighted edges in the graph.  To see this, notice that, for a tree,
the number of connected components of $G_+$ is exactly one plus the
number of negative edges: if we start with a tree with only positive
edges, every flip of an edge from positive to negative will add
another component to the positive subtree.  Conversely, since the
negative edges cannot form a cycle, the number of connected components
of $G_-$ is the number of vertices minus the number of negatively
weighted edges.  In short, if $G$ is a tree, then $n_+(G)$ is
independent of the weights of the edges, and very easy to compute by
inspection.

Now, consider the case where the graph has one cycle and one negative
edge.  It is easy to see in this case that the upper bound
in~\eqref{eq:bound} is one, and the lower bound is either zero (if the
negative edge lies in the cycle); or one (if the negative edge does
not lie in the cycle).  In the latter case, the number of positive
eigenvalues is again independent of the weights.  In the former case,
the weight on the edge matters, and it is not hard to imagine that the
relative magnitude of the weight on the negative edge compared to the
other edges in the cycle determines the stability.

From these examples, it is at least plausible to conjecture that the
stability computation is something that can done on a cycle-by-cycle
basis.  This is further supported by previous work of the
authors~\cite{Bronski.DeVille.14}, where it was shown that a certain
quotient of $H_1(G)$ by a subgroup was important for understanding the
spectrum of the graph Laplacian with positive and negative edges.

In fact, this is the result of the current paper.  Let us denote the
number of cycles by $C = |E|-|V|+1$.  Then the main result of this
paper is that the computation of the index of $\L_G$ is the equivalent
to computing the index of a particular $C\times C$ matrix (what we
call the ``cycle intersection matrix'').  If $C\ll |V|$, then this is
clearly a much simpler problem; the technique we describe here is most
effective for sparse graphs.

\subsection{Connection to nonlinearly coupled phase oscillators}

As stated above, the problem we consider is a pure linear algebra
problem and can be stated as such, but in fact it is inspired and
informed by considering its ``killer app'', i.e. understanding the
dimension of the unstable manifold of a fixed point of a nonlinearly
coupled network dynamical system.

Some examples of the kinds of models we are interested in include the 
Kuramoto problem on a general weighted connected graph $D = \{\delta_{ij}\}$, namely:
\begin{equation}\label{eq:Kuramoto}
  \frac{d\theta_i}{dt}=\omega_i + \sum \delta_{ij} \sin(\theta_j - \theta_i).
\end{equation}
More generally one can also consider the Kuramoto--Sakaguchi
model\cite{Kuramoto.Sakaguchi.1986,  DeSmet.Aeyels.07}
\begin{equation}\label{eq:KuramotoSakaguchi}
  \frac{d\theta_i}{dt}=\omega_i + \sum \delta_{ij} \sin(\theta_j - \theta_i+\alpha),
\end{equation}
as well as the swing equations, which govern the dynamics of a coupled system 
of generators and loads\cite{Anderson.Fouad.1977}
\begin{equation}\label{eq:swing}
  M  \frac{d^2\theta_i}{dt^2} + D \frac{d\theta_i}{dt}=\omega_i + \sum \delta_{ij} \sin(\theta_j - \theta_i).
\end{equation}
In each of these examples the dimension of the unstable manifold to a 
fixed point of the equations is given by the number of positive 
eigenvalues of a graph Laplacian. 

The study of the fixed points of these systems, and their stability,
has lead to a huge number of papers in the
literature~\cite{Kuramoto.book, Kuramoto.91,BS, E, TOR,
  Hansel.Sompolinsky.92, HK, HLRS, S, Acebron.etal.05, VO1, VO2, MS1,
  MS2, Dorfler.Bullo.12, Dorfler.Chertkov.Bullo.13,
  Coutinho.Goltsev.Dorogotsev.Mendes.13, Dorfler.Bullo.14}.  It has
been widely observed that the behavior of solutions varies
dramatically as the topology of the graph changes: these models are
essentially trivial on a tree, but become more complex as the graph
becomes denser. This again suggests that the first homology group
$H_1(D)$, which counts the number of loops in the graph, should be
important

To date, most of the analytical work has focused on special graph
configurations with a high degree of symmetry. The most-studied case,
where the behavior of the model is very well understood, is the case
of ``all-to-all'' coupling, where the underlying graph is the complete
graph $K_n$, and all of the edge weights $\delta_{ij}$ are equal. In a
ground breaking paper~\cite{MS2}, Mirollo and Strogatz gave upper and
lower bounds (which differed by 1) on the dimension of the unstable
manifold to an arbitrary fixed point solution to this
model. D\"{o}rfler and Bullo~\cite{Dorfler.Bullo.2011} gave a
sufficient condition for the existence of a stable phase-locked
solution, and two of the current
authors~\cite{Bronski.DeVille.Park.12} gave an explicit formula for
the dimension of the unstable manifold. Similarly Verwoerd and Mason
give a proof of the existence of a critical coupling constant for the
cases of the complete\cite{VO1} and bipartite graphs, and one of the
current authors~\cite{DeVille.12} gives some results on the minimizers
and saddles for the Kuramoto on cyclic graphs. For general topologies,
it was first observed by Ermentrout~\cite{Ermentrout.1992} that if the
graph is connected and all of the links are short (the angle
differences are less than $\frac{\pi}{2}$) then the fixed point is
stable.

We show in this paper that the intuition put forward in the preceeding
paragraphs is correct. We show that, rather than computing the number
of positive eigenvalues of the Laplacian, an object that acts on the
vertex space of the graph, there is a dual calculation that can be
posed on the cycle space. This dual problem gives an alternative
method for computing the dimension of the unstable manifold to a
particular fixed point. This approach is particularly attractive for
problems in which the graph is relatively sparse, and the dimension of
the cycle space is much smaller than the dimension of the vertex
space. We note that there are some very important applications for
which this condition is satisfied. One example is the graphs
associated with electrical power networks. Such networks are governed
by the swing equations~\eqref{eq:swing} and these tend to be
reasonably sparse. The following data is from Cotilla-Sanchez, Hines,
Barrows and Blumsack\cite{Cotilla-Sanchez.2012}, who calculated the
number of edges and vertices for the IEEE 300 model (a standard test
power system) as well as the Eastern (EI), Western (WI) and Texas (TI)
interconnects, the largest components of the North American electrical
network:\footnote{A small amount of processing has been done on this
  data, including combining parallel edges and removing disconnected
  buses. Other sources give slightly different numbers without
  changing the general picture. }

\begin{table}[ht]

\caption{Sparsities of Various Electrical Networks}

\centering
\begin{tabular}{|c|| c| c| c| c|}
\hline
\hline 
Network & Edges $E$ & Vertices $N$ & Number of Cycles $C=E-N+1$ & Sparsity $\frac{C}{N}$ \\ \hline\hline
IEEE 300 & 409 & 300  & 110 & .37 \\ \hline 
EI & 52075 & 41228 & 10848  & .26 \\ \hline
WI & 13734 & 11432 & 2303 & .2 \\ \hline 
TI & 5532 & 4513 & 1020 & .23 \\
\hline\hline
\end{tabular}
\label{table:ElectricalNetworks}
\end{table}
As can be seen from the data above electrical networks tend to be
quite sparse, with the dimension of the cycle space being 20---40\% of
dimension of the vertex space. For instance if one wanted to
understand the stability of a fixed point of the Texas Interconnect,
the dual problem presented in this paper would reduce the calculation
to the eigenvalues of a $1020\times 1020$ matrix, rather than the
$4513\times 4513$ eigenvalue problem determined by a naive stability
calculation.

\section{Summary of Results}

In this section, we define our notation and state the two main
theorems of the paper, and give a sense of the main ideas of the
proofs. The complete proof of the main results will be deferred to
Section~\ref{sec:proofs}.

\subsection{Definitions}

\begin{define}\label{def:graph}
  A {\bf weighted graph} $G$ is the triple $G=(V,E,\G)$, where $V$ is
  a set of vertices, $E\subseteq V\times V$ a set of edges, and $\G =
  \{\gamma_{e}\}_{e\in E}$ a set of edge weights.  By this definition,
  the graphs in this paper are always simple (at most one edge between
  vertices) and undirected.  
\end{define}

\begin{define}\label{def:Laplacian}
  Let $G=(V,E,\G)$ be a weighted graph.  The {\bf Laplacian matrix} of
  $G$ is the $|V|\times|V|$ matrix whose entries are given by
  \begin{equation*}
    (\L_G)_{v,w} = \begin{cases} \q\q\gamma_{(v,w)},&v\neq w,\\ -\sum_u \gamma_{(v,u)},& v=w.\end{cases}
  \end{equation*}
  Since the graph is undirected, the Laplacian matrix is always
  symmetric and has real spectrum.  We denote by $n_+(G), n_0(G),
  n_-(G)$ the number of positive, zero, and negative eigenvalues of
  $\L_G$, and refer to this triple as the {\bf spectral index} of $G$.
  Since, for any $G$, $\L_G{\bf 1} = {\bf 0}$, $\L_G$ has a zero
  eigenvalue and thus a zero determinant.  We define the {\bf reduced
    determinant} of $\L_G$, $\detred(\L_G)$, to be the determinant of
  the matrix restricted to the space of zero sum vectors, i.e. to the
  orthogonal complement of ${\bf 1}$ in $\R^{|V|}$.  (Note that
  $\detred(\L_G) = 0$ iff $\L_G$ has a nonsimple eigenvalue at $0$.)
  We will colloquially refer to a Laplacian as {\bf stable} if
  $n_+(\L_G) = 0$ and {\bf unstable} otherwise.
\end{define}

\begin{define}\label{def:incidence}
  Let $G = G(V,E,\G)$ be a connected weighted graph.  Choose and fix
  and orientation for each edge (so that each edge has both a head and
  tail vertex).  Define the {\bf signed incidence matrix} of $G$ as
  the $|V|\times |E|$ matrix whose entries are given by
\begin{equation*}
  (B_G)_{v,e} = \begin{cases}  1, & e = (v,*),\\ -1, & e = (*,v),\\ 0,& \mbox{else.}\end{cases}
\end{equation*}
The nullspace of the matrix is called the {\bf cycle space}.  It is a
standard result~\cite{Godsil.Royle} that if $G$ is connected, then the dimension
of this nullspace is $C:= |E(G)|-|V(G)|+1$.  We call this integer the
{\bf cycle rank} of the graph.  Note also that while the definition of
$B$ depends on the choice of orientation, the number $C$ does not.

Using the Rank--Nullity Theorem, the above is equivalent to saying
that $B_G^\top$ has nullity one, a fact we also use below.  It is also
not hard to see that $B_G^\top{\bf 1} = {\bf 0}$, so in fact
$\ker(B_G^\top) = \Span({\bf 1})$ whenever $G$ is connected.

We define the {\bf tree set} $\tree(G)$ to be the set of edges of $G$
that are contained in every spanning tree of $G$. The complement of
$\tree(G)$ is the {\bf cycle set} $\cycle(G)$, consisting of every
edge which is included in { some cycle} in the graph.  In terms of
the incidence matrix, $e$ is in the cycle set if and only if there is
a vector $x\in N(B_G)$ with $x_e\neq 0$.
\end{define}

\begin{define}\label{def:cycle}
  Let $y_1,y_2,\dots, y_C$ be a basis for the cycle space.  Define the
  $|E|\times C$ matrix $Y_G$ whose columns are given by the $y_i$, and
  $D_G$ to be the $|E|\times|E|$ diagonal matrix where the $e$th entry
  is $\gamma_e$.  Finally, we define the {\bf cycle intersection
    matrix}, or the {\bf cycle form}
  \begin{equation}\label{eq:defofZ}
    \mcZ_G:= -Y_G^\top D_G^{-1}Y_G,
  \end{equation}
  i.e. the matrix with components
  \begin{equation*}
    (\mcZ_G)_{ij} = -\sum_{e\in E}\gamma_e^{-1}{y_{i,e} y_{j,e}}.
  \end{equation*}
  Note that $\mcZ_G$ is always a $C\times C$ matrix.  The exact form
  of $\mcZ_G$ will depend on the choice of basis, and we will specify
  the exact choice of basis in~\eqref{eq:defofY} below.  For now, we
  simply note that any property that we need for $\mcZ_G$ will be
  independent of this choice.
\end{define}

\begin{lem}\label{lem:Laplacian-decomposition}
  Let $G=(V,E,\Gamma)$ be a connected weighted graph.  Define $D$ as
  in Definition~\ref{def:cycle} and $B$ as in Definition~\ref{def:incidence}.  Then
  \begin{equation*}
    \L_G = -B_GD_GB_G^\top.
  \end{equation*}
\end{lem}

\begin{proof}
  Let us first assume that $v\neq w$.  Then
\begin{equation*}
  (B_GD_GB_G^\top)_{v,w} = \sum_{e,e'} B_{ve}D_{ee'}B^\top_{e'w} = \sum_e B_{ve}B_{we}\gamma_e.
\end{equation*}
If $v,w$ are adjacent, then either $e=(v,w)$ or $e=(w,v)$ is an
oriented edge.  In either case, $B_{ve}B_{we} = -1$.  Moreover, if $e$
is an edge not incident on both $v$ and $w$, then $B_{ve}B_{we} = 0$,
and thus the sum has a single term, and $(B_GD_GB_G^\top)_{vw} = -\gamma_e$.
Therefore all of the off-diagonal terms are the same as $-\L_G$.

Finally, notice that since $B_G^\top {\bf 1} = {\bf 0}$, we have
$B_GD_GB_G^\top{\bf 1} = {\bf 0}$, and from this it follows that the
diagonal terms must be correct as well.

\end{proof}

\subsection{Theorems}

The following are the main results of the paper. The first result gives 
a formula for the number of positive eigenvalues of the Laplacian in 
terms of the number of edges with negative edge weights together
with the number of positive eigenvalues of the cycle intersection 
matrix.

\begin{thm}\label{thm:index}  
  Let $G = (V,E,\G)$ be a connected, weighted graph and let $\L_G$ be
  its Laplacian matrix.  Then the number of positive eigenvalues of
  the Laplacian is the number of negative edges in the graph minus the
  number of positive eigenvalues of the cycle intersection matrix,
  i.e.
  \begin{equation}\label{eq:LZ}
    n_+(\L_G) = \#\{e\in E\,|\, \gamma_e<0 \} - n_+(\mcZ_G).
  \end{equation}
  Recall in the definition of $\mcZ_G$ that one makes a choice of
  basis; however, it follows from Lemma~\ref{lem:sylvester} that the
  index of $\mcZ_G$ is independent of this choice of basis.

In particular, we have the inequality
\begin{equation*}
  n_+(\L_G) \le \#\{e\in\tree(G) \,|\, \gamma_e<0 \}.
\end{equation*}
From~\eqref{eq:LZ}, the maximal number of negative edges a graph can
have and still have a stable Laplacian is the cycle number $C$.  In
the case where the Laplacian is stable and there are exactly $C$
negative edges, then there is choice of basis cycles so that there is
exactly one negative edge in each cycle in the basis. In other words
if there are exactly $C$ negative edges, then the graph with the
negative edges removed must be a tree.
\end{thm}

\begin{remark}
The basis for the cycle space can be chosen in the following way: Let $G_+$ be 
the subgraph containing only the edges of $G$ with positive weight. The cycle 
space of $G_+$ forms a subspace $S_+$ of the full cycle space of $G$. Similarly
let $G_-$  be the subgraph containing only negatively weighted edges, and 
$S_-$ the corresponding cycle space. Note that $S_+$ and $S_-$ are necessarily 
orthogonal. The basis for the cycle space of the full graph $G$ can be chosen 
as the union of bases for $S_+$ and $S_-$ together with a set of ``mixed'' 
cycles that cannot be decomposed into a sum of cycles containing only edges 
of one sign. These mixed cycles are a basis for the quotient of the full group 
of cycles by the subgroups $S_\pm$. In this basis the cycle intersection form 
can be written in the following partitioned form 
\[
\mcZ_G=\left(
    \begin{array}{c|cc}
      A_{M}&B_{M,+}&B_{M,-}\\\hline
      B^t_{M,+}&\mcZ_{G_+}& 0\\
      B^t_{M,-}&0&\mcZ_{G_-}
    \end{array}
  \right)
\label{eq:partition}
\] 
where $\mcZ_{G_+}$ and $\mcZ_{G_-}$ are the cycle intersection forms for the subgraphs $G_+$ and $G_-$ respectively. 

First we note that $\mcZ_{G_+}$ is negative definite and $\mcZ_{G_-}$
is positive definite, so that $n_+(\mcZ_G)$ is at least as large as
the dimension of the cycle space of $G_-$ and no larger than the
dimension of the full cycle space minus the dimension of the cycle
space of $G_+$. These inequalities are equivalent to the inequalities
derived previously by the authors\cite{Bronski.DeVille.14}.

Secondly we note that the calculation of the index of $\mcZ_G$ can be
further reduced to a calculation on the space of mixed cycles. In
particular by applying the Haynsworth theorem to \eqref{eq:partition}
we have that
\[
n_+(\mcZ_G) = n_+(A_{M} - B_{M,+}\mcZ^{-1}_{G_+} B^t_{M,+} - B_{M,-}\mcZ^{-1}_{G_-} B^t_{M,-}), 
\] 
and thus the calculation can be reduced to finding the index of a matrix 
whose size is the number of mixed cycles. However since the above formula 
requires inversion of the matrices $\mcZ_{G_\pm}$ it is not obvious that 
this is more attractive than working on the full cycle space.  
\end{remark}

The second theorem is a similar result at the level of determinants
rather than the spectral index. This result is not really needed for
the applications, but it is proven with similar ideas and provides a
nice dual form of the famous matrix tree theorem. This is essentially
a weighted version of a theorem proven by Sjogren,\cite{Sjogren.91}
though the proof given here is quite different.

\begin{thm}\label{thm:detred}
  Let $G = (V,E,\G)$ be connected, and define $\L_G$ as above.  Let
  $N_G = |V(G)|$, and then
\begin{equation}\label{eq:detred}
 \frac1 {N_G} \detred(\L_G) = \det(\mcZ(G)) \times \prod_{e \in E} \gamma_e. 
\end{equation}
The quantity $\det(\mcZ(G))$ is a homogeneous Laurent polynomial in
the edge weights which generates labelled tree complements: each term
in the Laurent polynomial represents a set of edges that, when removed
from the graph, yield a tree.
\end{thm}

\begin{remark}
  There is an obvious connection between the two theorems, but they
  are not quite equivalent since they address different quantities.
  For example, Theorem~\ref{thm:detred} seems stronger since it gives
  an actual number, but it only implies Theorem~\ref{thm:index} modulo
  2.
\end{remark}

We defer the proof of these results until a later section, but we will
briefly summarize the two main ideas here. The first is that for any
$G$, there is a {\bf covering tree} $T$ (which is a weighted tree),
and a graph map $\varphi\colon T\to G$, which is surjective on
vertices and bijective on edges.  In other words if one identifies
certain vertices of the tree $T$ one obtains $G$. The identification
of vertices in the graph corresponds at the level of linear algebra to
restriction of the Laplacian to a certain subspace, so the Laplacian
on any graph can be realized as the restriction of the Laplacian on a
tree to an appropriate subspace.
 
The second idea is to use the Haynsworth theorem\cite{H.1968}, which
is stated in Lemma~\ref{lem:haynsworth}.  The Haynsworth theorem is a
kind of duality result relating the number of positive eigenvalues of
a matrix to the number of positive eigenvalues of a $k\times k$
principle minor matrix and the number of positive eigenvalues of the
complementary $(n-k)\times(n-k)$ minor of the inverse matrix.  In
other words, modulo some invertibility assumptions, if $M$ is
self-adjoint and $S$ is a subspace of $\R^N$, then
\begin{equation*}
  n_+(M) = n_+(M|_S) + n_+(M^{-1}|_{S^\perp}).
\end{equation*}
Brushing aside some technical points associated with invertibility
which will be addressed below, the basic observation is that if we
take $M$ to be the Laplacian on the covering tree and $M|_S$ to be the
Laplacian on the graph $G$, then one can compute that
$M^{-1}|_{S^\perp}$ is exactly the cycle intersection form $Z_G$. For
a tree it is easy to see that the number of positive eigenvalues is
exactly the number of edges with negative edge-weights, and formula
\eqref{eq:LZ} follows. There is an analogous identity at the level of
determinants, giving \eqref{eq:detred}.

\section{Linear Examples}\label{sec:linex}

\subsection{Diamond graph}\label{sec:diamond}
We consider the simplest graph with cycle rank $C=2$, the diamond
graph shown in the left frame of Figure~\ref{fig:diamond.cover}, with
five edges and four vertices.  The Laplace matrix is given by
  \begin{equation*}
    \L_G 
    = \left(\begin{array}{cccc}
        -(\gamma_a+\gamma_d+\gamma_e)& \gamma_a & \gamma_e & \gamma_d \\ 
        \gamma_a & -(\gamma_a+\gamma_b) & \gamma_b & 0 \\
        \gamma_e & \gamma_b & -(\gamma_b+\gamma_c+\gamma_e) & \gamma_c \\
        \gamma_d & 0 & \gamma_c & -(\gamma_c+\gamma_d) 
      \end{array}\right)
\end{equation*}
By the matrix-tree theorem, we can obtain the eight labelled spanning
trees by considering the $4,4$ cofactor
\begin{equation}\label{eq:detLG}
\begin{split}
  \det\L_{44}(G)  
  &=\left|\begin{array}{ccc}
      -(\gamma_a+\gamma_d+\gamma_e)& \gamma_a & \gamma_e \\ 
      \gamma_a & -(\gamma_a+\gamma_b) & \gamma_b \\
      \gamma_e & \gamma_b & -(\gamma_b+\gamma_c+\gamma_e) 
    \end{array}\right|\\
  &=\gamma_a \gamma_b \gamma_c+\gamma_a \gamma_b \gamma_d+
  \gamma_a \gamma_c\gamma_d+\gamma_a \gamma_c \gamma_e+\gamma_a
  \gamma_d \gamma_e+\gamma_b \gamma_c \gamma_d+\gamma_b \gamma_c
  \gamma_e+\gamma_b\gamma_d \gamma_e.  
\end{split}
\end{equation}
Let us orient this graph clockwise around the outside, then $1\to3$.
With this orientation, a basis for the cycle space is given by
\begin{align*}
 z_1 &= (1,1,0,0,-1)^t \\
 z_2 &= (0,0,1,1,1)^t. 
\end{align*} 
The cycle intersection matrix $Z_G$ is given by 
\begin{equation*}
Z_G = -\left(\begin{array}{cc}\dfrac{1}{\gamma_a}+\dfrac{1}{\gamma_b}+\dfrac{1}{\gamma_e}& \dfrac{1}{\gamma_e}\\\\
 \dfrac{1}{\gamma_e} &\dfrac{1}{\gamma_c}+\dfrac{1}{\gamma_d}+\dfrac{1}{\gamma_e} \end{array}\right).  
\end{equation*}
The determinant of this matrix is
\begin{equation}\label{eq:detCG}
\det Z_G = \dfrac{1}{\gamma_d \gamma_e}+\dfrac{1}{\gamma_c \gamma_e}
+\dfrac{1}{\gamma_b\gamma_e}+\dfrac{1}{\gamma_b \gamma_d}
+\dfrac{1}{\gamma_b \gamma_c}+\dfrac{1}{\gamma_a \gamma_e}
+\dfrac{1}{\gamma_a \gamma_d}+\dfrac{1}{\gamma_a \gamma_c}.
\end{equation}
It is not hard to see that if we multiply~\eqref{eq:detCG} by
$\g_a\g_b\g_c\g_d\g_e$, then we obtain the expression
in~\eqref{eq:detLG}, verifying Theorem~\ref{thm:detred}. The $\det(Z_G)$ 
generates labelled tree complements: each term in this Laurent polynomial 
represents a pair of edges that, when removed from the graph, yield a 
tree.  
 
Moreover, we can use Theorem~\ref{thm:index} to compute bounds on the
edge weights to determine stability.  For ease of notation, we write
$\rho_i = 1/\gamma_i$, and
\begin{equation*}
  w_1 = \rho_a + \rho_b,\quad   w_2 = \rho_c+\rho_d,\quad   w_{12} = \rho_e.
\end{equation*}
Then we have
\begin{equation}\label{eq:diamond-zg}
  Z_G = -\left(\begin{array}{cc} w_1+w_{12}& w_{12}\\w_{12} &w_2+w_{12} \\\end{array}\right).
\end{equation}

Let us first consider the case where there is one negative edge.  In
this case, $Z_G$ has either one or zero positive eigenvalues, which is
the same as saying that $-Z_G$ has one or zero negative eigenvalues.
If there is one, then $\L_G$ is stable, and if zero, it is unstable.

Note however that $\det Z_G <0$ if and only if there is one positive
and one negative eigenvalue, which gives the condition
\begin{equation*}
  (w_1+w_{12})(w_2+w_{12}) - w_{12}^2 < 0,
\end{equation*}
or
\begin{equation*}
  w_1w_2 + w_1w_{12} + w_2w_{12} < 0.
\end{equation*}
If we assume that $\gamma_e<0$, then, solving for $w_{12}$, we obtain the inequality
\begin{equation*}
  w_{12} < -\frac{w_1w_2}{w_1+w_2},
\end{equation*}
or
\newcommand{\fg}[1]{\rho_{{#1}}}
\begin{equation*}
  \rho_e < -\frac{\fg a+\fg b+\fg c+\fg d}{\l(\fg a+\fg b\r)\l(\fg c+\fg d\r)}.
\end{equation*}
Notice, for example, that if we assume that $\gamma_a\approx 0$, then
the $\rho_a$ terms dominate, and the condition we obtain for stability
is $\gamma_e > -(\rho_c+\rho_d)$.  Here $\gamma_b$ is not relevant,
and this makes sense: as $\gamma_a\to 0$, the $a,b$ path is almost
broken as a loop and thus cannot significantly affect stability, so
the condition for stability will be a balance between $e$ and $c,d$.

On the other hand, if one of the edges on the outside is negative (say
wlog $\gamma_a$), we obtain the inequality 
\begin{equation*}
  w_1 < \frac{w_2w_{12}}{w_2w_{12}},
\end{equation*}
or
\begin{equation*}
  \fg a < -\dfrac{\l(\fg c + \fg d\r)\fg e}{\fg c + \fg d+ \fg e}-\fg b.
\end{equation*}
Note that every $\rho$ on the right-hand side is positive.  Note the
particular role played by the $\rho_b$ term: if we make $\gamma_b$
larger, this makes $\rho_b$ smaller (in absolute value), making it
more difficult for $\rho_a$ to cause the instability.

qqqqq

Now, assume the graph has two negative edges and let us try and find
conditions for stability.  Note that we cannot have both $a,b$ or both
$c,d$ negative, since this automatically gives an unstable Laplacian.
Therefore we need that at least two of $w_1,w_2,w_{12}$ be negative.

The condition that the matrix be stable is that $-Z_G$ have two
negative eigenvalues, which we can obtain by the condition of having a
positive determinant and a negative trace, so
\begin{equation*}
  w_1w_2 + w_1w_{12} + w_2w_{12} > 0,\quad w_1+w_2+2w_{12} < 0.
\end{equation*}
By symmetry, we only need to consider the case where $w_1,w_2<0$ and
where $w_1,w_{12}<0$.  In the former case, there is a negative edge on
each of the ``outsides'' of the graph, and in the latter there is one
in the middle.

\subsection{One-cycle graphs}\label{sec:1cycle}

Let us now consider the general one-cycle graph and consider the
condition for stability.  Note that such a graph can have at most one
negative edge and still be stable, and, moreover, it cannot be
contained in the tree part of the graph, but must be part of the
cycle.  Therefore we can simplify and consider $R_N$, the ring graph
on $N$ vertices: this graph has $N$ vertices and $N$ edges, and let us
label them so that edge $i$ goes from $i\to i+1$, modulo $N$. 

The cycle can be represented by the vector ${\bf 1}$, and thus we have
$Z_G$ is the $1\times1$ matrix (i.e. the scalar)
\begin{equation*}
  Z_G = -\sum_{i=1}^N \frac1{\gamma_i}.
\end{equation*}
Thus to obtain a stable Laplacian, we need that $Z_G > 0$, or 
\begin{equation*}
  \sum_{i=1}^N \frac1{\gamma_i} < 0.
\end{equation*}
Assuming that $\g_1<0$ and the remainder are positive, the condition
for stability is

\begin{equation*}
  \gamma_1 > \frac{-1}{\displaystyle\sum_{i=2}^n \dfrac1{\gamma_i}}.
\end{equation*}
If all of the positive edges have an equal weight, $\gamma_i \equiv
\gamma$, then we obtain
\begin{equation*}
  \gamma_1 > -\gamma/(n-1).
\end{equation*}
In this case, we see that increasing $\gamma$ helps with stability,
whereas increasing the number of edges, all things being equal, hurts
stability.

\subsection{Two-cycle graphs}\label{sec:2cycle}

We can consider the general two-cycle graph; in fact, we can reuse
many of the computations of Section~\ref{sec:diamond} above.

As in the previous section, we simplify by assuming that there are no
edges in $\tree(G)$, i.e. each edge is in one or more cycles.  Let us
split the edges into three sets $S_1,S_2,S_{12}$; the edges in $S_1$
are those only in cycle one, those in $S_2$ are those only in cycle
two, and finally $S_{12}$ are those edges both in cycles one and two.
It is clear that we can choose $S_1,S_2$ to be non-empty, but for some
graphs, it may then be the case that $S_{12}$ is empty (e.g. a graph
with a figure-eight structure).  Let us reorder the edges so that when
ordered, the sets $S_1,S_2,S_{12}$ come in order, and we write $k_* =
|S_*|$.  From this we see that we can choose the cycle basis as
\begin{equation*}
  y_1 = (1,1,\dots,1,0,\dots,0,1,\dots, 1),\quad   y_2 = (0,0,\dots,0,1,\dots,1,1,\dots, 1),
\end{equation*}
where the three groups of edges are of length $k_1,k_2,k_{12}$.  If we then define 
\begin{equation*}
  w_* = \sum_{e\in S_*} \frac1{\gamma_e},
\end{equation*}
then we see that the cycle form $Z_G$ has the same structure as
in~\eqref{eq:diamond-zg}, and the subsequent analysis holds over for
this case.

The simplest case is the one where $S_{12} = \emptyset$ and $w_{12} =
0$.  In this case $Z_G$ is diagonal, and the number of positive
eigenvalues is exactly the number of negative diagonal entries.  For
example, let us say that the edge $e$ for which $\gamma_e < 0$ is in
cycle $1$, and thus we need $w_1< 0$ to have a stable Laplacian.  This
gives
\begin{equation*}
  \frac1{\gamma_e} + \sum_{f\neq e,f\in S_1} \frac1{\gamma_f} < 0,
\end{equation*}
or
\begin{equation*}
  \gamma_e > \dfrac{-1}{\sum_{f\neq e,f\in S_1} 1/{\gamma_f}}.
\end{equation*}
This is exactly the condition as in the one-cycle case.  Similarly, if
the graph has two negative edges, we need both $w_1,w_2<0$ and the
conditions on each cycle are decoupled.

In the case where $w_{12}>0$, this gives an analysis similar to that
given above.  Let us consider the case where there is one negative
edge, and it is in both cycles one and two.  We again have the
condition
\begin{equation*}
  w_{12} < -\frac{w_1w_2}{w_1+w_2}.
\end{equation*}
If, for example, we consider the case where all of the positive edges
have the same weight $\gamma$, and the one negative edge has weight
$\gamma_e<0$, then we have
\begin{equation*}
  w_1 = k_1/\g,\quad w_2 = k_2/\g,\quad w_{12} = \frac{k_{12}-1}\gamma + \frac1{\gamma_e}.
\end{equation*}
This simplifies to
\begin{equation}\label{eq:geg}
  \frac1{\gamma_e} < -\frac 1\g \left(\frac{k_1k_2}{k_1+k_2} + (k_{12}-1)\right).
\end{equation}
Again, increasing $\gamma$ helps for stability since it makes the
condition easier to achieve, but increasing any of $k_1,k_2,k_{12}$
makes the condition harder to achieve.

The interpretation of these facts is as follows: clearly, increasing
the strength of the positive couplings will make it harder for the
negative coupling to destabilize the entire system, and in fact
from~\eqref{eq:geg} we see that the two weights are effectively
linearly related.  Moreover, we see that increasing the number of
links in either of the cycles makes it easier to destabilize the
system, similarly to the case of a single cycle.

\section{Nonlinear example --- Kuramoto model on a ring}\label{sec:nlex}

The examples in the previous section studied the stability of linear
Laplacians, but we can also use the current results to obtain novel
results for nonlinearly coupled oscillators as well --- in particular
we consider the Kuramoto model.

It has been known for a long time that if we have a fixed point for
the Kuramoto model where all of the links are ``short'', i.e. the
angle difference across a link is less than $\pi/2$ in absolute value,
then this gives a stable fixed point, see~\cite{Ermentrout.1992}.  We
could then pose the question of whether there exist solutions with
``long'' links that are still stable (this would correspond to having
negative off-diagonal terms in the Jacobian).  This question was
studied numerically in~\cite{Dorfler.Chertkov.Bullo.13}; there the
authors considered a large class of random networks and random
frequencies, and showed that stable fixed points with long links do
exist, but they are exceedingly rare (typically the proportion of
stable points found that have long links were less than $1/1000$).

Intuitively, it makes
sense that if we make the weight $\delta_{ij}$ on an edge arbitrarily
small, then this is possible, but we could then pose the question: can
we have a stable long-link solution if the edge weights are all equal?
We will show that this answer is affirmative on the ring by a
judicious choice of frequencies.  In particular, we will answer the
following question for the ring topology: given a specific graph
topology for the Kuramoto system with equal edge weights, what is the
largest possible link for a stable fixed point of the system?

We consider the Kuramoto model as in~\eqref{eq:Kuramoto}, where we
assume that the graph is a ring with $n$ vertices and that the edge
weights are all equal, and by rescaling we can assume that they are
all one, so that $\delta_{i,i+1} = 1$ (here and below this sum on the
indices is understood modulo $n$).  Let $\theta^*$ be a fixed point
of~\eqref{eq:Kuramoto}.  From this we have
\begin{equation}\label{eq:Knn}
   0 = \omega_i + \sin(\theta^*_{i+1}-\theta^*_i) + \sin(\theta^*_{i-1}-\theta^*_i),
\end{equation}
and the Jacobian at this point is given by $J$, where
\begin{equation*}
  J_{i,i\pm 1} = \cos(\theta_i^* - \theta_{i\pm 1}^*),
\end{equation*}
and the diagonal terms are chosen to make it a Laplacian.  Then we see
from the previous section that it is a stable fixed point if one of
the following two conditions are met:
\begin{enumerate}

\item For all $i=1,\dots, n$, $\gamma_{i,i+1} :=  \cos(\theta_i^\star -
  \theta_{i+1}^\star) > 0$;

\item one of the $\gamma_{i,i+1}$ is negative, and
\begin{equation}\label{eq:cossum}
  \sum_{i=0}^{n-1} \frac{1}{\gamma_{i,i+1}} = \sum_{i=0}^{n-1} \frac{1}{\cos(\theta_i^\star -
  \theta_{i+1}^\star)} < 0.
\end{equation}
\end{enumerate}

Note that $\cos(\theta)>0$ iff $\av{\theta} < \pi/2$, and thus the
sign of the cosine of the angle difference tells us whether a link is
long or short.  These conditions can be reinterpreted as saying: the
fixed point is stable if all of the links are short, or if there is
one long link, then we need the sum of the reciprocals of the edge
weights to be negative.  The question then remains: is there a
configuration $\theta^\star$ that is a fixed point
of~\eqref{eq:Kuramoto}, has one long link, and
satisfies~\eqref{eq:cossum}?

Let us write $\zeta_i = \theta_{i+1} -
\theta_i$.  Note that every term in the right-hand side
of~\eqref{eq:Knn} will be a $\zeta_i$, except for the term
$\theta_{n-1}-\theta_0$ which we can write as $\sum_{i=0}^{n-2}
\zeta_i$.  Then~\eqref{eq:cossum} becomes
\begin{equation}\label{eq:cossum2}
  \sum_{i=0}^{n-1} \frac1{\cos(\zeta_i)} + \frac{1}{\cos\left(\sum_{i=0}^{n-1} \zeta_i\right)}<0.
\end{equation}
Without loss of generality, let us assume that the longest link is
given by the gap $\theta_{n-1} - \theta_0$, so that if there is one
long link, it is represented by the last term in~\eqref{eq:cossum2}. 

We can then pose this problem in an alternate form as follows: for any
$\eta \in (0,2\pi)$, let us find the maximum of
\begin{equation*}
  f(\zeta) = \sum_{i=0}^{n-2} \frac1{\cos(\zeta_i)},
  \mbox{ subject to }
  g(\zeta) = \sum_{i=0}^{n-2} \zeta_i = \eta.
\end{equation*}
Then it is possible to realize a fixed point with largest gap $\eta$ iff
\begin{equation*}
  \min_\zeta f(\zeta) + \frac1{\cos\eta} < 0.
\end{equation*}
The nice thing we see from the constrained optimization problem is
that since the constraint is symmetric up to any permutation of the
variables, any internal optimizer must be as well.  (We could also use
Lagrange multipliers to obtain the same result.)  Thus the maximum of
$f$ is acheived only when all of the $\zeta_i$ are equal,
and~\eqref{eq:cossum2} becomes
\begin{equation}\label{eq:cond2}
  \frac{n-1}{\cos(\zeta)} + \frac{1}{\cos((n-1)\zeta)} < 0.
\end{equation}
Summarizing our conditions, if we assume that $\cos(\zeta)$ is positive, then iff the conditions
\begin{equation*}
  \cos((n-1)\zeta) > 0\mbox{ or } (\cos((n-1)\zeta)< 0 \mbox{ and } \frac{n-1}{\cos(\zeta)} + \frac{1}{\cos((n-1)\zeta)} < 0
\end{equation*}
hold, then we have a stable fixed point for the solution $\theta_k^* =
k\zeta$.  Note that both of these conditions can be simplified to
require that their product be positive, i.e.
\begin{equation}\label{eq:hn}
  h_n(\zeta):= (n-1)\frac{\cos(n-1)\zeta}{\cos(\zeta)} + 1 > 0.
\end{equation}
To exhibit a ``long link'' solution, we need to show that there is
$\zeta^*$ with roots of $h_n(\zeta^*)=0$ and $|n\zeta^* \pmod 2\pi|
\in (\pi/2,3\pi/2)$.  We will show numerically that this is possible
for every $n$, see Table~\ref{tab:zeta}.

\begin{table}[ht]
\begin{centering}
\begin{tabular}{|c|c|}\hline
 $n$ & $n\zeta^*/2\pi$ \\\hline
 3 & 0.447 \\\hline
 4 & 0.392 \\\hline
 5 & 0.358 \\\hline
 10 & 0.297 \\\hline
 20 & 0.272 \\\hline
 30 & 0.264 \\\hline
 40 & 0.261 \\\hline
 50 & 0.258 \\\hline
\end{tabular}
\caption{Length of longest possible link for a stable solution on the
  ring with $n$ vertices}
  \label{tab:zeta}
\end{centering}
\end{table}

Here we have found the first root of $h_n(\zeta)$ and compute
$n\zeta_n$: this is the distance between $\theta_0$ and $\theta_{n-1}$
if we choose the solution $\theta_k = k\zeta^*$.  Note that in each
case, the length of the link is greater than a quarter turn (note that
we are presenting the length divided by $2\pi$, so that a link is
``long'' iff it is greater than $1/4$), but this distance decreases
monotonically to $1/4$ as $n\to\infty$.  

Note that in general, $h_n(\zeta)$ has multiple roots, so that the
first root may not always give the largest link, but we have found
empirically that this is so.  As such, the numbers in
Table~\ref{tab:zeta} represent the longest link possible for a stable
fixed point of the Kuramoto system on a ring of size $n$.

We can also compute the values of $\omega_k$ that give this solution.
For all $k=1,\dots,n-2$, we have $\theta_{k+1}-\theta_k =
\theta_{k-1}-\theta_k = \zeta$, so that $\omega_k = 0$.  However, we
also have
\begin{equation*}
  \omega_{n-1}(\zeta):= - \sin(\theta_{0}-\theta_{n-1}) - \sin(\theta_{n-2}-\theta_{n-1}) = \sin((n-1)\zeta) + \sin(\zeta),
\end{equation*}
and $\omega_0 = -\omega_{n-1}$.

\begin{figure}[ht]
\begin{centering}
  \includegraphics[width=0.75\textwidth]{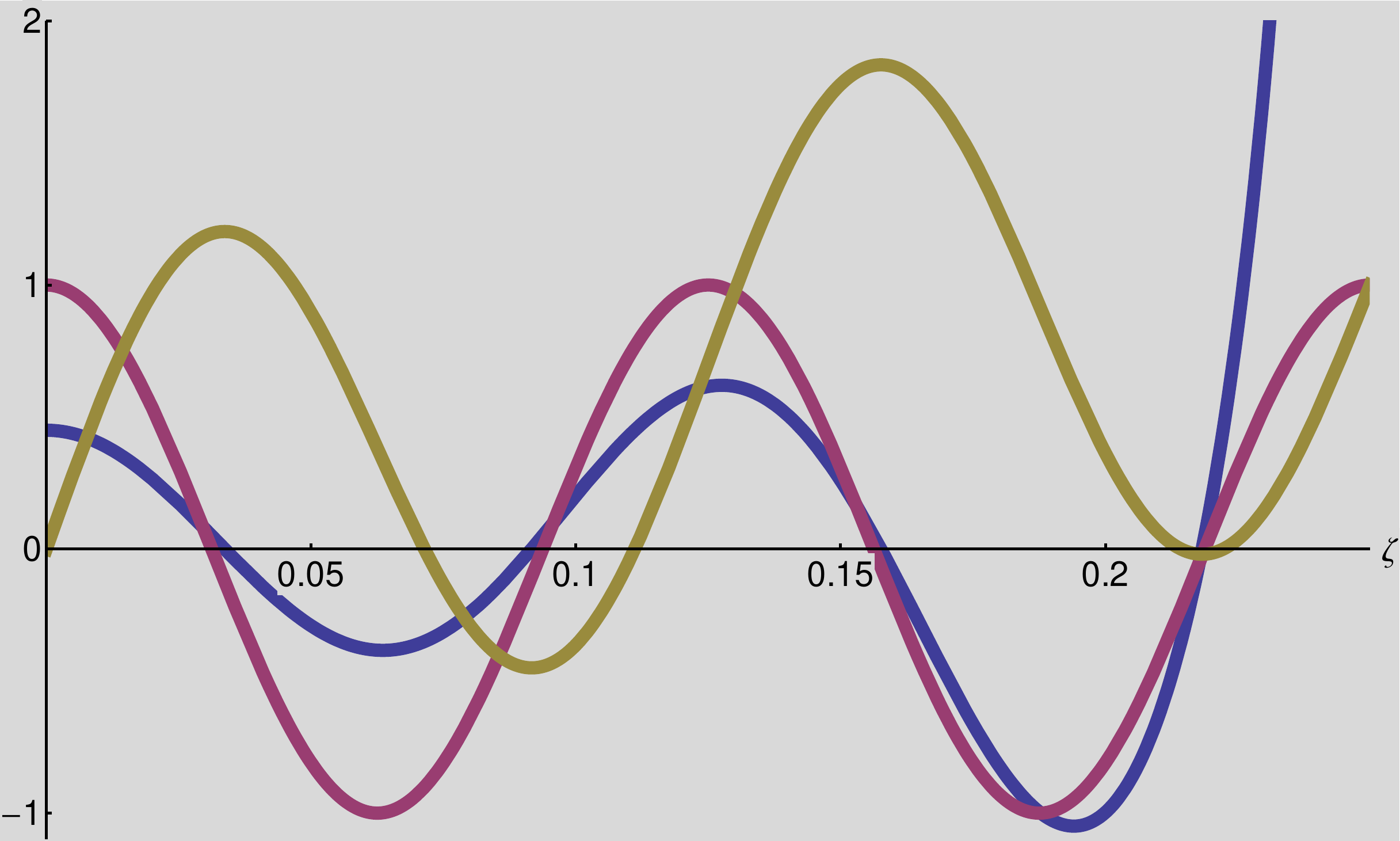}
  \caption{We plot several of the functions referenced in the text for
    $n=9$ vertices in the ring: $0.05 * h_9(\zeta)$ (blue), $\omega_8(\zeta)$
    (orange), $\cos(8\zeta)$ (purple).}
  \label{fig:curves}
\end{centering}
\end{figure} 

We plot several of the curves for the case of $n=9$ in
Figure~\ref{fig:curves}.  We have scaled $h_9(\zeta)$ so that all the
curves are of the same order of magnitude; note that only its sign
matters.  As shown above, where the blue curve is positive, the
solution is stable.  The purple curve is the cosine of the length of
the $(n-1,0)$ link: since we see that this becomes negative before
$h_9(\zeta)$ does, there is a stable long link configuration.
Finally, note that the zeros of $h_9(\zeta)$ correspond to the
critical points of $\omega_8(\zeta)$ (this can be checked), and thus
the $\omega_8(\zeta)$ curve can be interpreted as a series of stable
and unstable branches: whenever $\omega_8(\zeta)$ is increasing in
$\zeta$, this corresponds to a stable solutions, and when it is
decreasing, this corresponds to a $1$-saddle.  As we slide up and down
in this graph, this leads to saddle--node bifurcations in the obvious
manner.  Moreover, note that the $\omega=0$ slice is the one
considered in~\cite{DeVille.12} --- it was shown there that for $n=9$
there are two nontrivial sinks and two corresponding saddles, and this
is recovered here.  By decreasing or increasing $\omega_7$, we can
cause these saddles to collide with either of two sinks.

\section{Proof of Theorems}\label{sec:proofs}

\subsection{Proof of Theorem~\ref{thm:index}}

\subsubsection{Definition of covering tree}

\begin{define}
  Let $G = (V,E,\G)$ be a connected graph.  We define the {\bf
    universal cover} of $G$, denoted $U(G)$.  If $G$ is a tree, then
  $U(G) = G$.  If $G$ is not a tree, then $U(G)$ is a countably
  infinite tree defined as follows:

\begin{itemize}

\item Choose a root vertex $r\in V(G)$.

\item The vertices of $U(G)$ are {\em non-backtracking rooted paths}
  in $G$, i.e. sequences of the form $w=(r,v_2,v_3,\dots, v_n)$ where
  $(v_i,v_{i+1})\in E$, and $v_{i-1}\neq v_{i+1}$ for all $i$.

\item We say that two vertices $w_1,w_2\in V(U(G))$ are adjacent if
  one is an extension of the other, i.e.
\begin{equation*}
  w_1 = (r,v_2,v_3,\dots, v_n),\quad w_2 = (r,v_2,v_3,\dots, v_n,v_{n+1}),
\end{equation*}
or vice versa.

\item Finally, we define the weight of the edge $(w_1,w_2)$ in the
  example above to be the same as the weight on the edge
  $(v_{n-1},v_n)$.

\end{itemize}
\end{define}

\begin{define}[Fundamental domain]\label{def:T}
  For any graph $G=(V,E,\G)$, we define $T$, a {\bf fundamental domain
    of the cover} of $G$, as follows.  First construct the universal
  cover $U(G)$.  Choose $T$ to be any finite connected subtree of
  $U(G)$ such that each edge of $G$ appears once.  Notice that by
  definition, $G$ and $T$ have the same edge sets, so we will write
  $E$ for either set.  Since there are $|E|$ edges in $T$ and it is a
  tree, then $|V(T)=|E|+1$.  We will also use the notation $N_G =
  |V(G)|$ and $N_T = |V(T)|$ throughout.
\end{define}

\begin{remark}  The universal cover has certain properties:

\begin{enumerate}

\item If $G$ is not a tree, then $U(G)$ has countably many vertices.
  The degree of the node $w = (r,v_2,\dots, v_n)$ in $U(G)$ is the
  same as the degree of $v_n$ in $G$.  In particular, $U(G)$ is
  locally finite.

\item The map 
\begin{align*}
  \varphi\colon U(G) &\to G,\\
  (r,v_2,\dots,v_n)&\mapsto v_n
\end{align*}
is a covering map, i.e. it is a map of graphs, it is surjective, and
it is an isomorphism in the neighborhood of any vertex.

\item The construction of $T$ is independent (up to isomorphism) of
  the choice of the root vertex $r$.

\item Finally, it should be noted that we can construct $T$ directly
  without using the universal cover, see Example~\ref{exa:diamond} for
  an example.

\end{enumerate}
\end{remark}
\begin{prop}\label{prop:surjection}
  Let $G=(V,E,\G)$ be a connected weighted graph, and define $T$
  as above.  Then the graph map $\varphi\colon T\to G$ given by
\begin{equation*}
  \varphi((r,v_2,\dots, v_n)) = v_n
\end{equation*}
is a surjective map.  Note that $|V(T)| - |V(G)| = C$, i.e. the
difference in the number of vertices of a fundamental domain of the
covering tree and the original graph itself is the cycle rank of the
graph.
\end{prop}

\begin{proof}
  Assume that $w_1,w_2\in V(T)$ and $(w_1,w_2)\in E(T)$.
  Since $w_1,w_2$ are adjacent in $U(G)$, we have that
\begin{equation*}
  w_1 = (r,v_{1},\dots, v_{k}),\quad   w_2 = (r,v_{1},\dots, v_k, v_{k+1}),
\end{equation*}
and $(v_k,v_{k+1})\in E(G)$.  Then $\varphi(w_1) = v_k$ and
$\varphi(w_2) = v_{k+1}$ are adjacent in $G$ whenever $w_1,w_2$ are
adjacent in $T$.  Therefore $\varphi$ is a graph homomorphism.  To see
that it is surjective, note that every edge of $G$ appears in $T$, and
therefore its incident vertices must as well.  Finally, note that
\begin{equation*}
  |V(T)| - |V(G)| = |E| + 1 - |V(G)| = C.
\end{equation*}

\end{proof}

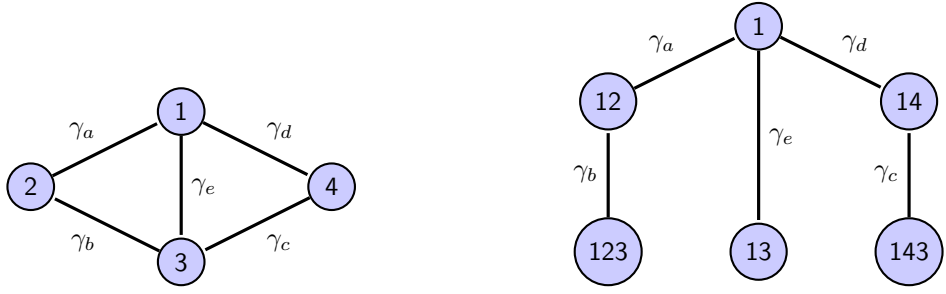
\begin{figure}[ht]
\begin{center}
\begin{tikzpicture}[>=stealth',shorten >=1pt,auto,node distance=2cm,
  thick,main node/.style={circle,fill=blue!20,draw,font=\sffamily}]
  \node[main node] at (0,1) (1) {1};
  \node[main node] at (2,0) (4) {4};
  \node[main node] at (-2,0) (2) {2};
  \node[main node] at (0,-1) (3) {3};
  \path[very thick] (2) edge node  {$\gamma_a$} (1) 
        (1) edge node {$\gamma_e$} (3)
        (3) edge node  {$\gamma_b$} (2)
        (4) edge node  {$\gamma_c$} (3)
        (1) edge node {$\gamma_d$} (4);
\end{tikzpicture}
\hspace{1in}
\begin{tikzpicture}[>=stealth',shorten >=1pt,auto,node distance=2cm,
  thick,main node/.style={circle,fill=blue!20,draw,font=\sffamily}]
  \node[main node] at (0,1) (1) {1};
  \node[main node] at (2,0) (14) {14};
  \node[main node] at (-2,0) (12) {12};
  \node[main node] at (-2,-2) (123) {123};
  \node[main node] at (2,-2) (143) {143};
  \node[main node] at (0,-2) (13) {13};
  \path[very thick] (12) edge node  {$\gamma_a$} (1) 
        (1) edge node {$\gamma_e$} (13)
        (123) edge node  {$\gamma_b$} (12)
        (143) edge node  {$\gamma_c$} (14)
        (1) edge node {$\gamma_d$} (14);
\end{tikzpicture}
\end{center}
\caption{The diamond graph $G$, and a choice of fundamental domain of
  its universal cover $T$.  Note that $G$ and $T$ have the same edge
  sets with five edges, and $T$ necessarily has six vertices.}
\label{fig:diamond.cover}
\end{figure}

\begin{example}\label{exa:diamond}
  Let us consider the diamond graph in Figure~\ref{fig:diamond.cover}.
  If we choose the root vertex as $1$, then $2,3,4$ are neighbors and
  so we add $12,13,14$ to $T$.  Since $\gamma_b,\gamma_c$ have not yet
  appeared, we add them as in the tree.  Of course, this is not the
  only covering tree we can obtain, even assuming that we choose $1$
  as the root; for example, if we remove vertex $143$ and add a vertex
  $1234$, this would also be a covering tree.
\end{example}

\subsubsection{The projections}

\begin{define}\label{def:X}
  Let $G$ be a connected graph, and $T$, $\varphi\colon T\to G$ as
  defined above.  We define $X$ to be the $N_T\times N_G$ matrix whose
  elements are
\begin{equation*}
  X_{a,v} = 1(\varphi(a)=v).  
\end{equation*}
We will also refer to the columns by the vectors $x^v,v\in V(G)$.
\end{define}

\begin{lem}\label{lem:Xproj}
  Let $G=(V,E,\G)$ be a connected weighted graph, and $T, X$ as in
  Definitions~\ref{def:T},~\ref{def:X}.  Then
  \begin{equation}\label{eq:vw1}
    \L_G = X^\top \L_T X,
  \end{equation}
  or
  \begin{equation}\label{eq:vw2}
    (\L_G)_{v,w} = \ip{x^v}{\L_T x^w}.
  \end{equation}
\end{lem}

\begin{proof}
  Let us use formula~\eqref{eq:vw2}.  Writing this out, we have
\begin{equation*}
  \ip{x^v}{\L_Tx^w} = \sum_{i,j} x^v_i (\L_T)_{ij} x^w_j = \sum_{i,j}(\L_T)_{ij} 1_{\varphi(i)=v}1_{\varphi(j)=w}.
\end{equation*}
If we can show that there exists a unique pair $i,j$ such that
$\phi(i) = v $, $\phi(j) = w,$ and $(\L_T)_{ij} = \gamma_{(v,w)}$,
then we are done.  But notice that since $\phi(i) = v$, and $\phi(j) =
w$, if $(i,j)\in E(T)$, then the weight on the edge $(i,j)$ must be
$\gamma_{(v,w)}$, since $\phi$ is a graph homomorphism.  Moreover,
since each edge of $G$ can appear at most once in $T$, this means that
there is only one such pair.
\end{proof}

\begin{define}\label{def:Y}
  The map $\phi\colon T\to G$ is surjective but in general not
  injective.  For each $v\in V(G)$ with $\av{\phi^{-1}(v)}>1$, choose
  one representative of $\phi^{-1}(v)$ as the {\bf primary
    representative} of this set.  Now, for each $w\in\phi^{-1}(v)$
  with $w\neq \prim(v)$, define the vector $q^w\in \R^{N_T}$ by
\begin{equation*}
  q^w_i = \begin{cases} -1, &i=\prim(v),\\ 1, &i=w,\\ 0,&else.\end{cases}
\end{equation*}
We define $Q$ as the matrix whose columns are the $q^w$.
\end{define}

\begin{lem}\label{lem:Sbasis}
  Let $S\subseteq \R^{N_T}$ be the span of the $x^v$ defined in
  Definition~\ref{def:X}.  Then the vectors $q^w$ form a basis for
  $S^\perp$.  In particular, $\dim(\Span(q^w)) =\rank(Q)= c$, the
  cycle rank of $G$.
\end{lem}

\begin{proof}
  First we see that each $q^w$ is orthogonal to $S$.  Consider $x^v$
  where $v=\phi(w)$.  This vector is equal to $1$ both in the
  $\prim(v)$ and $w$ index, and therefore $\ip{q^w}{x^v} = 0$.
  Moreover, if we consider $x^v$ with $v\neq \phi(w)$, then this
  vector is equal to zero on these indices, and again $\ip{q^w}{x^v} =
  0$.  Since $\dim(\Span(S)) = \rank(X) = |V(G)|$, so that $\rank(Q) =
  |E|+1-\dim S = |E| - |V(G)| +1 = c$.  Finally, note that the $q^w$
  are clearly independent by construction, so by dimension counting
  they form a basis for $S^\perp$.

\end{proof}

\begin{lem}\label{lem:z}
  The Laplacian $\L_T\colon \R^{N_T}\to\R^{N_T}$ is not invertible,
  but has a one-dimensional nullspace given by $\Span({\bf 1})$.  In
  this case, the canonical map
  $$\widetilde\L_T\colon \R^{N_T}/\{{\bf 1}\}\to\R^{N_T}/\{{\bf 1}\} $$
  is invertible.  We then have that
  \begin{equation*}
    Q^\top (\widetilde\L_T)^{-1} Q = \mcZ_G.
  \end{equation*}
\end{lem}

\begin{proof}
In general, for any graph $H$, the incidence matrix $B_H$ is map from
$\R^{|E(H)|}$ to $\R^{|V(H)|}$.  In particular, since $|E(T)| =
|E(G)|$ and $|V(T)| = |E(G)|+1$, this means that
\begin{equation*}
  B_{T}\colon \R^{|E(T)|} \to \R^{|V(T)|}
\end{equation*}
is a map of co-rank one, and as such $B_{T}^\top$ is a map with
a one-dimensional nullspace, which is $\Span({\bf 1})$.  Therefore we
can define $\widetilde{B}$ to be the canonical map so that
\begin{equation*}
  \widetilde{B}\colon \R^{|E(T)|}\to \R^{|E(T)|+1}/\{{\bf 1}\}
\end{equation*}
is invertible: to define this, we first define $\widetilde{B}^\top$ to
be the induced map on the quotient and then $\widetilde{B} =
(\widetilde{B}^\top)^\top$.  In particular, note that $\widetilde{B} =
B$ on ${\bf 1}^\perp$.

For each $w$ such that $q^w\neq \0$, consider the unique path in
$T$ from $w$ to $\prim(\phi(w))$.  Let us write this path as 
\begin{equation*}
  w=v_0\to v_1\to v_2\to\dots\to v_{n-1} \to v_n = \prim(\phi(w)).
\end{equation*}
Define a vector $r^w\in\R^{|E(T)|}$ as follows:
\begin{equation*}
  r^w_e = \begin{cases} 1,& e=(v_i,v_{i+1}),\exists i,\\ -1 ,& e=(v_{i+1},v_{i}),\exists i,\\0,&\mbox{else.}\end{cases}
\end{equation*}
We first compute that 
\begin{equation*}
  B_{T}r^w = q^w.
\end{equation*}
To see this, note that $B_{T}r^w = 0$ on any vertex not in the
path, and, moreover, will be zero on any vertex on the interior of the
path due to cancellation.  We need only consider the endpoints.  

First, assume that $w\to v_1$ is the chosen orientation for the edge.
Then $(B_{T}r^w)_w= ((B_{T})_{(w,v_1)}r_{(w,v_1)} = 1\cdot
1 = 1$.  If, on the other hand, $v_1\to w$ is the orientation, then we
have $(B_{T}r^w)_w= ((B_{T})_{(w,v_1)}r_{(w,v_1)} =
-1\cdot -1 = 1$ as well.  Similarly,
$(B_{T}r^w)_{\prim(\phi(w))} = -1$.

From this, it follows that if we define  
\begin{equation}\label{eq:defofY}
  Y = \widetilde{B}_{T}^{-1}Q,
\end{equation}
then the columns of $Y$ form a basis for the cycle space of $G$.

By Lemma~\ref{lem:Laplacian-decomposition}, $\L_T =
-B_{T}DB_{T}^\top$, and it is easy to see that similarly,
$\widetilde{\L}_T = -\widetilde B_{T}D\widetilde
B_{T}^\top$.  From this we have
\begin{equation*}
  Q^\top (\widetilde\L_T)^{-1} Q = -Q^\top \widetilde B^{-\top}_{T}D\widetilde
B_{T}^{-1}Q = -Y^\top D Y,
\end{equation*}
giving the definition~\eqref{eq:defofZ}.

\end{proof}

{\bf Proof of Theorem~\ref{thm:index}.}  Recall the definition of
$M:=\widetilde{\L}(T)$.  From Lemma~\ref{lem:Xproj} we have that
\begin{equation*}
  \widetilde{\L}_G = X^\top \widetilde{\L}_T X,
\end{equation*}
and Lemma~\ref{lem:sylvester} gives that
\begin{equation*}
  n_+(\L_G) = n_+(M|_S),
\end{equation*}
since the columns of $X$ form a basis for $S$.  From
Lemma~\ref{lem:z} we have that
\begin{equation*}
  \mcZ_G = Q^\top (\widetilde{\L}_T)^{-1} Q,
\end{equation*}
and Lemma~\ref{lem:sylvester} again gives that
\begin{equation*}
  n_+(\mcZ_G) = n_+((M)^{-1}|_{S^\perp}),
\end{equation*}
since the columns of $Q$ form a basis for $S^\perp$.  Using
Lemma~\ref{lem:haynsworth} gives
\begin{equation*}
  n_+(\L_G) = n_+(\L_T) - n_+(\mcZ_G).
\end{equation*}

Recall that $T$ is a tree.  Using Theorem~2.10
of~\cite{Bronski.DeVille.14}, this implies that the number of positive
eigenvalues of $\L_T$ are given exactly by the number of negative
edges in $E(T)$.  Recalling that the edge sets of $G$ and $T$ are the
same, we are done.

Finally, to prove the last sentence of the theorem, we again appeal to
Theorem~2.10 of~\cite{Bronski.DeVille.14}.  One consequence of this
theorem is that for $\L_G$ to be stable, $G_+$ must be connected.  If
the number of negative edges in $G$ is the same as the number of
cycles, this means that $G_+$ is a tree, and thus is a spanning tree
of $G$.  For each negative edge in $G$, consider the two incident
vertices and the unique path in $G_+$ connecting them.  This plus the
addition of the negative edge make a cycle in $G$, and it is clear
that each negative edge will be in a different cycle.  Using this
choice of cycles, we now have exactly one negative edge in each
cycle.\qed

\subsection{Proof of Theorem~\ref{thm:detred}}

We again start with several lemmas.

\begin{define}\label{def:U}
  Recall $G,T,\phi\colon T\to G$.  For each $v\in V(G)$, if
  $|\phi^{-1}(v)| = 1$, let $\psi(v) = \phi^{-1}(v)$; else, $\psi(v) =
  \prim(v)$ as defined in Definition~\ref{def:Y}.  Then order the
  vertices of $V(T)$ so that we list all of the vertices in the range
  of $\psi$ first, and then all of the other vertices in some order.
  It is easy to see that the range of $\psi$ as $N_G$ elements.

  We then define the first $N_G$ columns of the same way as we did in
  the definition of $X$, i.e. for $v\in V(G)$, define the $v$th column of $U$, $u^v$, by
  \begin{equation*}
    (u^v)_a = \begin{cases} 1,& \phi(a) = v,\\ 0,& \mbox{else},\end{cases}
  \end{equation*}
  and for columns $i=N_G+1,\dots, N_T$, define $u^i$ to be the
  standard basis function with a 1 in the $i$th slot.  By
  construction, $U$ is a lower triangular matrix with $1$s on the
  diagonal, so $\det(U) = 1$.
\end{define}

\begin{lem}\label{lem:ul}
  The $N_G\times N_G$ upper-left block of $U^\top \L_{T} U$ is
  $\L_G$.
\end{lem}

\begin{proof}
  The proof is more or less the same as Lemma~\ref{lem:Xproj}.  Notice
  that once we have chosen this ordering, the first $N_G$ columns of
  $U$ is the matrix $X$ as defined in Definition~\ref{def:X} with the
  corresponding ordering of vertices.
\end{proof}

\begin{lem}\label{lem:tensors}
  Let $H$ be a symmetric matrix with a one-dimensional nullspace, and
  let $x$ span this nullspace.  Then for any $y,z$, we have
\begin{equation*}
  \det(H+y\otimes z) = \frac{\ip x y \ip x z}{\ip x x}\detred H.
\end{equation*}

\end{lem}

\begin{proof}
  Since $H$ is self-adjoint, it has real eigenvalues and a orthonormal
  eigenbasis.  Let us write $\lambda_1 = 0$ with $\lambda_2,\dots,
  \lambda_n$ not zero, and let $w_1,\dots,w_n$ be the corresponding
  eigenvectors.  Let $W$ be the matrix whose columns are the $w_i$;
  since the eigenvectors form an orthonormal basis, $W^{-1} = W^\top$.

  Let $M_1$ be the diagonal matrix whose diagonal entries are the
  $\lambda_i$, and define $M_2 = (W^\top y)\otimes (W^\top z)$, i.e. 
  \begin{equation*}
    (M_2)_{ij} = \ip {w_i} y \ip {w_j} z.
  \end{equation*}
  Notice that $H = WM_1W^\top$, since this is just its
  diagonalization, and
\begin{equation*}
  WM_2W^\top = W((W^\top u)\otimes (W^\top w))W^\top = WW^\top(u\otimes w) WW^\top = u\otimes w,
\end{equation*}
where we have used the identity $(Ax)\otimes (By) = A(x\otimes
y)B^\top$.  Thus we have 
\begin{equation*}
  H + u\otimes w = W(M_1 + M_2)W^\top,
\end{equation*}
and since $W$ is unitary this gives
\begin{equation*}
  \det(H+ u\otimes w)  = \det(M_1+M_2).
\end{equation*}
To compute the last term, notice that we can use the multilinearity of
the determinant by rows and break up this into $2^n$ terms.  However,
note two things: the matrix $M_2$ is rank one and thus has only one
linearly independent row, so if we choose more than one row from
$M_2$, the determinant is zero.  Also, the first row of $M_1$ is zero,
so to obtain a nonzero determinant, we must choose the first row of
$M_2$.  Therefore the only term which gives a nonzero determinant is
the one where we choose the first row from $M_2$ and all of the other
rows from $M_1$.  Expanding by minors in the first column, the
determinant of this matrix is
\begin{equation*}
  \ip{w_1}y\ip{w_1}z\lambda_2\times\dots\times \lambda_n.
\end{equation*}
Since we chose $\lambda_1 = 0$, this means that $w_1$ is the unit
vector that spans $\ker(H)$, and also notice that $\detred H$ is by
definition $\lambda_2\times\dots\times \lambda_n$.  Thus we have
\begin{equation*}
  \det(H+y\otimes z)= \ip{w_1}y\ip{w_1}z\detred H
\end{equation*}
If $x$ is a general vector in $\ker(H)$, then $x$ is a scalar multiple
of $w_1$ and we have
\begin{equation*}
  \det(H+y\otimes z)= \frac{\ip{x}y\ip{x}z}{\ip xx}\detred H.
\end{equation*}

\end{proof}

{\bf Proof of Theorem~\ref{thm:detred}.}  Recall the definition of $U$
in Definition~\ref{def:U}.  Define
\begin{equation*}
  A = U^\top (\L_{T} + \1\otimes \1)U = U^\top \L_{T} U + U^\top \1 \otimes U^\top \1.
\end{equation*}
If we can establish the following three identities:
\begin{align}
  \det(A) &= N_T^2 \prod_{e\in E(G)} \gamma_e,\label{eq:A}\\
  \det(A|_S) &= \frac{N_T^2}{N_G} \detred \L_G,\label{eq:As}\\
  \det((A^{-1})|_{S^\perp}) &= \det(\mcZ_G)\label{eq:Ai},
\end{align}
then, using Lemma~\ref{lem:HJ}, we obtain 
\begin{equation*}
  N_T^2 \prod_{e\in E(G)} \gamma_e = \frac{\dfrac{N_T^2}{N_G} \detred \L_G}{\det(\mcZ_G)},
\end{equation*}
and this proves the theorem.  So it remains to show these three identities.

Identity~\eqref{eq:Ai} follows directly from Lemma~\ref{lem:z}.  
To establish~\eqref{eq:A}, we note that
\begin{equation*}
  \det(A) = \det(U^\top (\L_{T} + \1\otimes \1)U) = \det(\L_{T} + \1\otimes \1),
\end{equation*}
since $\det(U) = 1$.  By Lemma~\ref{lem:tensors}, we have
\begin{equation*}
  \det(\L_{T} + \1\otimes \1) = \frac{\ip\1\1\ip\1\1}{\ip\1\1}\detred\L_{T} = N_T\detred\L_{T}, 
\end{equation*}
and since there is only one spanning tree of $T$, by the Matrix Tree
Theorem,
\begin{equation*}
  \detred\L_{T} = N_T \prod_{e\in E(G)} \gamma_e.
\end{equation*}
To establish~\eqref{eq:As}, note that if we use the basis for
$\R^{N_T}$ made from the columns of $U$, then $A|_S$ is just the upper
left $N_G\times N_G$ block of $A$, which is itself the upper left
block of $U^\top \L_{T} U + U^\top \1 \otimes U^\top \1$.  By
Lemma~\ref{lem:ul}, the upper left block of the first term is just
$\L_G$, and the upper left block of the second term can be written as
\begin{equation*}
  \widetilde{U}^\top \1 \otimes \widetilde{U}^\top\1,
\end{equation*}
where $\widetilde(U)$ is just the first $N_G$ columns of $U$.  In
particular, the $v$th row of $\widetilde{U}^\top \1$ is a zero-one
vector with as many ones as vertices in $\phi^{-1}(v)$, and therefore
$u: =\widetilde{U}^\top \1$ is a column vector of height $N_G$ whose $v$th
entry is $\av{\phi^{-1}(v)}$.  

Putting this together, and again using Lemma~\ref{lem:tensors} gives
\begin{equation*}
  \det(A|_S) = \det(\L_G + u\otimes u) = \frac{\ip u\1\ip u \1}{\ip\1\1}\detred \L_G = \frac{\left(\sum_{v\in V(G)}\phi^{-1}(v)\right)^2}{N_G}\detred \L_G = \frac{N_T^2}{N_G}\detred \L_G.
\end{equation*}

\qed

\appendix

\section{Classical lemmae}

This follows from the following duality formula for the number of
negative eigenvalues (or more generally the inertia) of a matrix. This
result seems to have been frequently rediscovered but the first
instance we have found in the literature is due to
Haynsworth~\cite{H.1968}.  This result has been used frequently in
the nonlinear waves literature, in the context of the stability of
traveling wave solutions.

\begin{lem}[Haynsworth]\label{lem:haynsworth}
  Suppose that $M$ is a non-singular $N\times N$ Hermitian matrix.
  Let $S$ a subspace of subspace of $\R^N$, $S^\perp$ the orthogonal
  subspace, $P_S$ the orthogonal projection onto $S$, and $M|_S=P_S
  MP_S$ the restriction of $M$ to $S$.  If $M|_S$ is non-singular, then
\begin{equation*}
  n_+(M) = n_+(M|_S) + n_+((M^{-1})|_{S^\perp})
  \end{equation*}

\end{lem}

Note that, in her original paper, Haynsworth states this theorem
slightly differently in terms of the Schur complement of the matrix,
but it is obviously equivalent.

\begin{lem}\label{lem:HJ}
  Suppose that $M$ is a non-singular $N\times N$ Hermitian matrix.
  Let $S$ a subspace of subspace of $\R^N$, $S^\perp$ the orthogonal
  subspace, $P_S$ the orthogonal projection onto $S$, and $M|_S=P_S
  MP_S$ the restriction of $M$ to $S$.  If $M|_S$ is non-singular, then
\begin{equation*}
  \det (M) = \frac{\det(M|_S)}{\det(M^{-1}|_{S^\perp})}.
\end{equation*}
\end{lem}

This is more or less the determinantal version of
Lemma~\ref{lem:haynsworth} and could be proved in a similar fashion.
It can also be found explicitly in~\cite{Horn.Johnson.book}.

\begin{lem}[Sylvester's Law of Inertia]\label{lem:sylvester}
  Suppose $M$ is Hermitian and $U$ is non-singular. Then
\begin{equation}
  n_-(M) = n_-(U^\top M U).
\end{equation}
In particular, if $A$ is an $n\times n$ matrix, $S\subset \R^n$ is a
$k$-dimensional subspace, $P_S$ is the orthogonal projection onto $S$,
$x_1,\dots, x_k$ is a basis for $S$, and $X = (x_1,\dots, x_k)$ is the
$k\times k$ matrix whose $i$th column is $x_i$, then
\begin{equation*}
  n_+(P_SAP_S) = n_+(X^\top A X).
\end{equation*}

\end{lem}

\begin{proof}
  The first statement is just the standard formulation of Sylvester's
  Theorem, so we will prove the second part.  Notice that if we choose
  $y^1,\dots, y^k$ to be an orthonormal basis for $S$, then $P_SAP_S =
  Y^\top A Y$.  Choosing $Z = Y^{-1}X$, we have
\begin{equation*}
  Z^\top (Y^\top A Y) Z = X^\top A X,
\end{equation*}
and by the first statement this implies that these two matrices have
the same signature.
\end{proof}


\begin{thebibliography}{10}

\bibitem{Bronski.DeVille.14}
Jared~C. Bronski and Lee DeVille.
\newblock Spectral theory for dynamics on graphs containing attractive and
  repulsive interactions.
\newblock {\em SIAM J. Appl. Math.}, 74(1):83--105, 2014.

\bibitem{BDK-SIMAX.15}
J.~Bronski, L.~DeVille, and P.~Koutsaki.
\newblock The spectral index of signed laplacians and their stability.
\newblock {\em arXiv:1503.01069}, 2015.

\bibitem{Kuramoto.Sakaguchi.1986}
H.~Sakaguchi and Y.~Kuramoto.
\newblock A {S}oluble {A}ctive {R}otater {M}odel showing {P}hase {T}ransitions
  via {M}utual {E}ntrainment.
\newblock {\em Prog. Theor. Phys.}, 76(3):576--581, 1986.

\bibitem{DeSmet.Aeyels.07}
Filip~De Smet and Dirk Aeyels.
\newblock Partial entrainment in the finite kuramoto–sakaguchi model.
\newblock {\em Physica D: Nonlinear Phenomena}, 234(2):81 -- 89, 2007.

\bibitem{Anderson.Fouad.1977}
P.M. Anderson and A.A. Fouad.
\newblock {\em Power {S}ystems {C}ontrol and {S}tability}, volume~1.
\newblock Iowa State Press, 1977.

\bibitem{Kuramoto.book}
Y.~Kuramoto.
\newblock {\em Chemical oscillations, waves, and turbulence}, volume~19 of {\em
  Springer Series in Synergetics}.
\newblock Springer-Verlag, Berlin, 1984.

\bibitem{Kuramoto.91}
Y.~Kuramoto.
\newblock Collective synchronization of pulse-coupled oscillators and excitable
  units.
\newblock {\em Physica D}, 50(1):15--30, May 1991.

\bibitem{BS}
Neil~J. Balmforth and Roberto Sassi.
\newblock A shocking display of synchrony.
\newblock {\em Phys. D}, 143(1-4):21--55, 2000.
\newblock Bifurcations, patterns and symmetry.

\bibitem{E}
G.~Bard Ermentrout.
\newblock Synchronization in a pool of mutually coupled oscillators with random
  frequencies.
\newblock {\em J. Math. Biol.}, 22(1):1--9, 1985.

\bibitem{TOR}
Dane Taylor, Edward Ott, and Juan~G. Restrepo.
\newblock Spontaneous synchronization of coupled oscillator systems with
  frequency adaptation.
\newblock {\em Phys. Rev. E (3)}, 81(4):046214, 8, 2010.

\bibitem{Hansel.Sompolinsky.92}
D.~Hansel and H.~Sompolinsky.
\newblock Synchronization and computation in a chaotic neural network.
\newblock {\em Phys. Rev. Lett.}, 68(5):718--721, Feb 1992.

\bibitem{HK}
Seung-Yeal Ha, Eunhee Jeong, and Moon-Jin Kang.
\newblock Emergent behaviour of a generalized {V}iscek-type flocking model.
\newblock {\em Nonlinearity}, 23(12):3139--3156, 2010.

\bibitem{HLRS}
Seung-Yeal Ha, Corrado Lattanzio, Bruno Rubino, and Marshall Slemrod.
\newblock Flocking and synchronization of particle models.
\newblock {\em Quart. Appl. Math.}, 69(1):91--103, 2011.

\bibitem{S}
S.~H. Strogatz.
\newblock From {K}uramoto to {C}rawford: exploring the onset of synchronization
  in populations of coupled oscillators.
\newblock {\em Phys. D}, 143(1-4):1--20, 2000.

\bibitem{Acebron.etal.05}
J.A. Acebr{\'o}n, L.L. Bonilla, C.J.P. Vicente, F.~Ritort, and R.~Spigler.
\newblock The {K}uramoto model: A simple paradigm for synchronization
  phenomena.
\newblock {\em Rev. Mod. Phys.}, 77(1):137, 2005.

\bibitem{VO1}
M.~Verwoerd and O.~Mason.
\newblock Global phase-locking in finite populations of phase-coupled
  oscillators.
\newblock {\em SIAM J. Appl. Dyn. Syst.}, 7(1):134--160, 2008.

\bibitem{VO2}
Mark Verwoerd and Oliver Mason.
\newblock On computing the critical coupling coefficient for the {K}uramoto
  model on a complete bipartite graph.
\newblock {\em SIAM J. Appl. Dyn. Syst.}, 8(1):417--453, 2009.

\bibitem{MS1}
R.~Mirollo and S.H. Strogatz.
\newblock The {S}pectrum of the {P}artially {L}ocked {S}tate for the {K}uramoto
  model.
\newblock {\em Journal of Nonlinear Science}, 17(4):309--347, 2007.

\bibitem{MS2}
R.~E. Mirollo and S.~H. Strogatz.
\newblock Synchronization of pulse-coupled biological oscillators.
\newblock {\em SIAM J. Appl. Math.}, 50(6):1645--1662, 1990.

\bibitem{Dorfler.Bullo.12}
Florian Dorfler and Francesco Bullo.
\newblock Synchronization and transient stability in power networks and
  nonuniform kuramoto oscillators.
\newblock {\em SIAM Journal on Control and Optimization}, 50(3):1616--1642,
  2012.

\bibitem{Dorfler.Chertkov.Bullo.13}
F.~D\"{o}rfler, M.~Chertkov, and F.~Bullo.
\newblock Synchronization in complex oscillator networks and smart grids.
\newblock 110(6):2005--2010, 2013.

\bibitem{Coutinho.Goltsev.Dorogotsev.Mendes.13}
BC~Coutinho, AV~Goltsev, SN~Dorogovtsev, and JFF Mendes.
\newblock Kuramoto model with frequency-degree correlations on complex
  networks.
\newblock {\em Physical Review E}, 87(3):032106, 2013.

\bibitem{Dorfler.Bullo.14}
Florian D{\"o}rfler and Francesco Bullo.
\newblock Synchronization in complex networks of phase oscillators: A survey.
\newblock {\em Automatica}, 50(6):1539--1564, 2014.

\bibitem{Dorfler.Bullo.2011}
Florian D{\"o}rfler and Francesco Bullo.
\newblock On the critical coupling for {K}uramoto oscillators.
\newblock {\em SIAM J. Appl. Dyn. Syst.}, 10(3):1070--1099, 2011.

\bibitem{Bronski.DeVille.Park.12}
J.~C. Bronski, L.~DeVille, and M.~J. Park.
\newblock Fully synchronous solutions and the synchronization phase transition
  for the finite-{N} {K}uramoto model.
\newblock {\em Chaos}, 22(033133), 2012.

\bibitem{DeVille.12}
Lee DeVille.
\newblock Transitions amongst synchronous solutions in the stochastic kuramoto
  model.
\newblock {\em Nonlinearity}, 25(5):1473, 2012.

\bibitem{Ermentrout.1992}
G.~Bard Ermentrout.
\newblock Stable periodic solutions to discrete and continuum arrays of weakly
  coupled nonlinear oscillators.
\newblock {\em SIAM J. Appl. Math.}, 52(6):1665--1687, 1992.

\bibitem{Cotilla-Sanchez.2012}
E.~Cotilla-Sanchez, P.D.H. Hines, C.~Barrows, and S.~Blumsack.
\newblock Comparing the topological and electrical structure of the north
  american electric power infrastructure.
\newblock {\em IEEE Systems Journal}, 6(4), 2012.

\bibitem{Godsil.Royle}
C.~Godsil and G.~Royle.
\newblock {\em Algebraic graph theory}, volume 207 of {\em Graduate Texts in
  Mathematics}.
\newblock Springer-Verlag, New York, 2001.

\bibitem{Sjogren.91}
Jon~A. Sjogren.
\newblock Cycles and spanning trees.
\newblock {\em Math. Comput. Modelling}, 15(9):87--102, 1991.

\bibitem{H.1968}
Emilie~V. Haynsworth.
\newblock Determination of the inertia of a partitioned {H}ermitian matrix.
\newblock {\em Linear Algebra and Appl.}, 1(1):73--81, 1968.

\bibitem{Horn.Johnson.book}
Roger~A. Horn and Charles~R. Johnson.
\newblock {\em Matrix analysis}.
\newblock Cambridge University Press, Cambridge, 1985.

\end{thebibliography}

\end{document}